\documentclass{article}

\usepackage{arxiv}
\usepackage{inputenc}
\usepackage{amssymb,amsmath}
\usepackage{amsfonts}
\usepackage{amsthm}
\usepackage{cite,bm}
\usepackage[pdftex]{graphicx}

\newtheorem{lemma}{Lemma}
\newtheorem{theorem}{Theorem}

\newtheorem{corollary}{Corollary}

\newtheorem{example}{Example}

\title{Unification in pretabular extensions of S4}

\author{Stepan~I. Bashmakov\thanks{This work is supported by the Krasnoyarsk Mathematical Center and financed by the Ministry of Science and Higher Education of the Russian Federation in the framework of the establishment and development of regional Centers for Mathematics Research and Education (Agreement No. 075-02-2020-1534/1).}\\
Institute of Mathematics and Computer Science\\
Siberian Federal University\\
Svobodny 79, Krasnoyarsk, 660041,
Russia \\ 
\texttt{ krauder@mail.ru}}

\begin{document}
\maketitle


\begin{abstract}
 L.L. Maksimova and L. Esakia, V. Meskhi showed that the modal logic $\mathcal{S}4$ has exactly 5 pretabular extensions PM1--PM5. In this paper, we study the problem of unification for all given logics. We showed that PM2 and PM3 have finitary, and PM1, PM4, PM5 have unitary types of unification. Complete sets of unifiers in logics are described.
\end{abstract}

\keywords{ pretabular logic \and Kripke semantic \and unification \and ground unifier \and projective formula \and unitary \and finitary}

\section*{Unification: tasks and methods}

The unification problem, apparently, was first investigated in the works of J. Robinson \cite{Robinson} in developing a resolution method for  of automatic proof systems. Having gone through a way of half a century development, the problem has become self-sufficient theory, actively studied in parallel in the areas of system programming \cite{Zakharov} and non-classical logic \cite{Balbiani}. In this paper, we consider exactly last field, where the main thesis of the theory was transformed into a statement about the possibility of turning formula into a theorem by replacement of variables.

Interest in the study of unification today includes establishing the unifiability of formulas and determining the boundaries of unifiability, searching for effective algorithms for constructing unifiers (and determining their forms) for formulas, as well as determining the best such substitutions, determining the type of unification in logic, as well as a number of related tasks.

One of the key modern methods is the algebraized approach proposed by S.~Ghilardi \cite{Ghil1, Ghil2} through projective approximation, which allowed describing the complete sets of formula unifiers quite efficiently. A few years before --- in 1992 --- S. Burris had proved the unitary unification for logics whose algebras contain a discriminatory term \cite{Burris}. As was shown later \cite{Dz2011}, actually S. Burris proved a projective unification in such logics, studied in detail by S.~Ghilardi \cite{Ghil1} and actively developed in many subsequent works \cite{DzW, IE2016} (including the papers of the author, \cite{BKR2, BKR2017, Bash2018}).

The projectivity of unification in logic allowed to prove the "good" types of unification in many logics and by many authors. One of the alternative approaches to describe complete sets of unifiers was proposed by V. Rybakov using adaptation of his method of $n$-characteristic models \cite {R2011cu}, successfully applied to solve the admissibility problem in a number of non-classical logics, starting in 1984 \cite{R1984, R1984FL}. Both of these approaches are reflected in this paper.

\section*{Pretabular Extensions of $\mathbf{\mathcal{S}4}$}

By a Kripke scale, we will standardly understand a pair $F = \langle W,R \rangle$, where $W$ is a basic set, and $R$ is a binary relation on $W$. The Kripke model $M := \langle F,\upsilon \rangle$ is defined as a scale with a valuation $\upsilon: Prop\mapsto 2^W$, where Prop is a countable set of propositional variables, all of the basic logical operators have their usual meaning, and $\forall x\in F, \forall \upsilon:$ 
$$(\langle F,x \rangle \Vdash_{\upsilon} \Box\varphi) \Leftrightarrow \forall y (x R y \Rightarrow \langle F,y \rangle\Vdash_{\upsilon} \varphi).$$

A formula $\varphi$ is \textit{valid}: 
\textit{on the model} $M := \langle F,\upsilon \rangle$, if $\forall x\in W$ $\langle F,x \rangle \Vdash_{\upsilon} \varphi$; 
\textit{on the scale} $F = \langle W,R \rangle$, if for any model $\langle F,\upsilon \rangle$ and $\forall x\in W$ $\langle F,x \rangle \Vdash_{\upsilon} \varphi$.

A class of scales is called \textit {characteristic} for a logic $\mathcal{L}$ iff all theorems of a logic are valid on all scales from this class.
A logic  $\mathcal{L}$ is called \textit{tabular}, if it can be characterized by a finite class of finite scales. A logic $\mathcal{L}$ is  \textit{pretabular}, if it is not tabular, but any of its proper extension is tabular.

The fundamental role in our research belongs to the well-known and studied logic $\mathcal{S}4$:
$$\mathcal{S}4 := K + (\Box p \rightarrow p) + (\Box p \rightarrow \Box\Box p)$$
which can be semantically characterized as the logic of all reflexive transitive scales. Of course, $\mathcal{S}4$ is not tabular. The tabularity problem is decidable over $\mathcal{S}4$ \cite{MaxVor2003}. 

Many well-known modal systems are extensions of $\mathcal{S}4$. These include the largest modal partner of intuitionistic logic $\mathcal{I}nt$ --- the Grzegorczyk logic $\mathcal{G}rz$:
$$\mathcal{G}rz := \mathcal{S}4 + (\Box(\Box(p \rightarrow \Box p) \rightarrow p) \rightarrow  p),$$
as well as the linear extension of $\mathcal{S}4$ --- the logic

$$\mathcal{S}4.3 := \mathcal{S}4 + \Box(\Box p \rightarrow q) \vee \Box(\Box q \rightarrow p).$$


In 1975, L.L. Maksimova \cite{Maks1975}, and independently in 1977 L.L.~Esakia and V.Yu.~Meskhi~\cite{EM1977} investigated pretabular extensions of $\mathcal{S}4$: it was proved that there are exactly 5 such logics, all of them are finitely approximable, Kripke complete, and of course axiomatizable. Following the notation proposed by L.L. Maksimova, we denote these pretabular logics $\mathbf{PM1}$--$\mathbf{PM5}$. 

\begin{figure}[h] \centering
	\includegraphics[width=11cm]{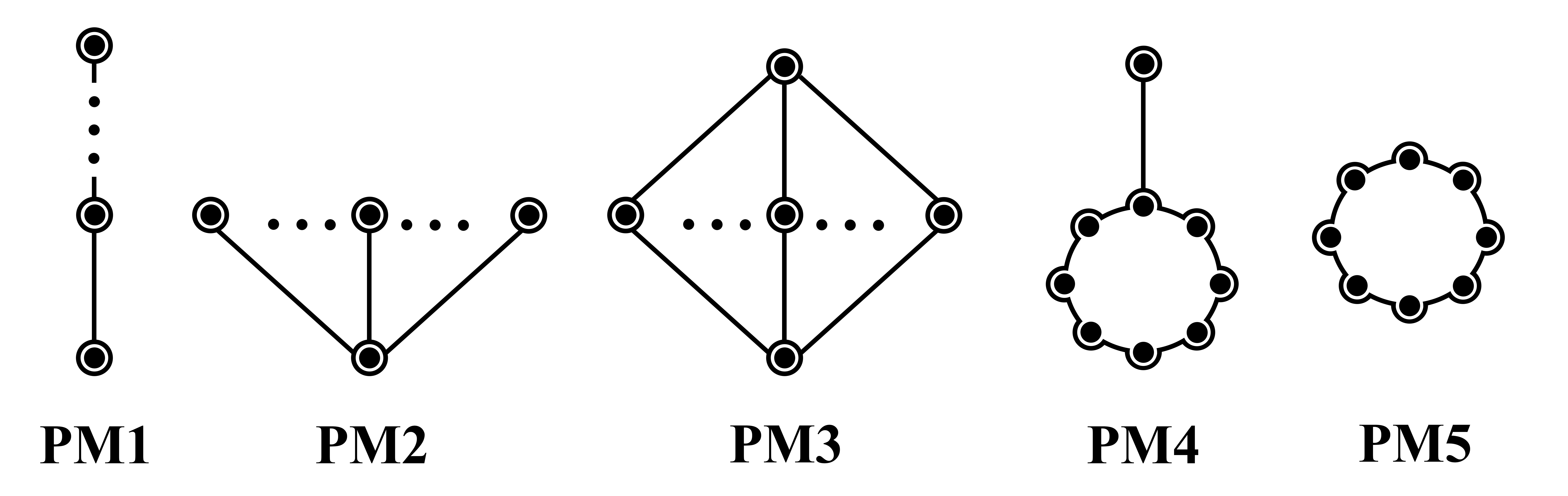}
	\caption{The scales of pretabular logics $\mathbf{PM1}$--$\mathbf{PM5}$.}
\end{figure}

$\mathbf{PM1}$--$\mathbf{PM5}$ are characterized by the classes of partially ordered scales. Further, also follow to the notation proposed by L.L. Maksimova \cite{Maks1975, Maks2016}, we will change only some indices in order to avoid misunderstandings when defining the $n$-characteristic model.

In reasoning and statements that are valid for all pretabular systems under consideration, for generality, we will use the common notation $\mathcal{L}$, setting $\mathcal{L}\in \{\mathbf{PM1}, \dots, \mathbf{PM5}\}\subset Ext(\mathcal{S}4)$.

\begin{enumerate}

\item $\mathbf{PM1} := \mathcal{S}4.3 + \mathcal{G}rz := [\Box (\Box (p \rightarrow \Box p)\rightarrow p) \rightarrow p].$

The logic $\mathbf{PM1}$ characterized by the class of linear scales $\mathbb{Z}_m$ \cite{Maks2016}:

\begin{center}
$\mathbb{Z}_m = \langle Z_m,R \rangle$, where $Z_m = \{1,\dots,m\}$, and $x R y \Leftrightarrow x \leq y$.
\end{center}

\item $\mathbf{PM2} := \mathcal{G}rz + \sigma_{2}:=[\Box p \vee \Box(\Box p \rightarrow \Box q \vee \Box\Diamond\neg q)].$

The logic $\mathbf{PM2}$ characterized by the class of partly ordered scales not containing 3-element chains \cite{Maks2016}:

\begin{center}
$\mathbb{V}_m = \langle V_m,R \rangle$, where $V_m = \{0,1,\dots,m\}$, and $x R y \Leftrightarrow (x = 0 \vee x=y)$.
\end{center}

\item $\mathbf{PM3} := \mathcal{G}rz + [\Box r \vee \Box (\Box r \rightarrow \sigma­_{2})] + (\Box\Diamond p \Leftrightarrow \Diamond\Box p).$

The characteristic class of scales for the logic $\mathbf{PM3}$ consists of partially ordered scales having the largest cluster (that is, the singular cluster of the first layer) and not containing 4-element chains \cite{Maks2016}:

\begin{center}
$\mathbb{U}_{m+1} = \langle U_{m+1},R \rangle$, where $U_{m+1} = \{0,1,\dots,m+1\}$, and $x R y \Leftrightarrow ([x=0] \vee [y=m+1] \vee [1\leq x=y\leq m])$.
\end{center}

\item $\mathbf{PM4} := \mathcal{S}4 + [\Box p \vee \Box(\Box p \rightarrow \Box q \vee \Box\Diamond\neg q)] + (\Box\Diamond p \leftrightarrow \Diamond\Box p).$

The logic $\mathbf{PM4}$ characterized by the class of partly ordered linear scales depth no more than 2 \cite{Maks2016} with the singular largest cluster and possible final cluster of the second layer: 

\begin{center}
$\mathbb{Y}_m = \langle Y_m,R \rangle$, where $Y_m = \{0,1,\dots,m\}$, and $x R y \Leftrightarrow ([x\leq m-1] \vee y=m)$.
\end{center}

\item $\mathbf{PM5} := \mathcal{S}5 = \mathcal{S}4 + [5:= p \rightarrow \Box\Diamond p].$

The logic $\mathbf{PM5}$ coincides with the modal system $\mathcal{S}5$ and is characterized by the class of scales of depth 1, which are cluster of a finite number of elements by equivalence \cite{Maks2016}:

\begin{center}
$\mathbb{X}_m = \langle X_m,R \rangle$, where $X_m = \{1,\dots,m\}$, and $\forall x,y: x R y$.
\end{center}

It is well known~\cite{Maks1975} that the logic $\mathbf{PM5}$ corresponds to the modal $\mathcal{S}5$, the pretabularity of which was established back in 1951, \cite{Scr}.

\end{enumerate}

An independent study of unification in the described pretabular logics was not carried out, although there are some correct conclusions transferring from known results obtained earlier for other cases of logics. We are interested in the systematization of the results, as well as in an autonomous study of unification for the cases of all pretabular logic $\mathbf{PM1}$--$\mathbf{PM5}$.

\section{Highlights of Unification Theory}

We introduce a number of basic definitions and known results of the unification problem that are used in further considerations. For a detailed study of unification in various non-classical logics, we recommend to turn to the monograph \cite{BaaSny}, articles \cite{Ghil1, Ghil2, BabR, BaaGhi, Jer_K, Dz_S5}.

A formula  $\varphi(p_1,\dots,p_s)$ called \textit{unifiable} in logic $\mathcal{L}$ iff there is a substitution $\sigma:$ $p_i \mapsto \sigma_i$, $\forall p_i \in Var(\varphi)$, s.t. $\varphi(\sigma_1,\dots,\sigma_s) \in \mathcal{L}$. In this case, $\sigma$ is an \textit{unifier} of $\varphi$.
\textit{Ground} unifier is a type of unifier, obtained by substituting the constants $\{\top, \bot\}$  instead of formula variables.

An unifier $\sigma$ of a formula $\varphi(p_1,\dots,p_s)$ called \textit{more general} than other $\sigma^1$ for $\varphi$ in $\mathcal{L}$ ($\sigma^1 \preceq \sigma$), if we can find a substitution $\sigma^2$, s.t. $\forall p_i\in Var(\varphi)$: $\sigma^1 (p_i)\equiv \sigma^2 (\sigma (p_i))\in \mathcal{L}$. Here $\preceq$ is a preorder on the set of all unifiers of the formula $\varphi$.

An unifier $\sigma$ of $\varphi(p_1,\dots,p_s)$ is \textit{maximal}, if for any other $\sigma^i$, either $\sigma^i \preceq \sigma$, 
or $(\sigma^i \npreceq \sigma) \& (\sigma \npreceq \sigma^i)$. If a formula has a single maximal unifier, it is called the \textit{most general} (\textit{mgu}, for short). A set of unifiers $CU$ for a formula $\varphi$ called \textit{complete} in $\mathcal{L}$, if for any unifier $\sigma$ of $\varphi$ there is $\sigma_1 \in CU$:  $\sigma \preceq \sigma^1$ (i.e. more general from $CU$).

Every unifier of a formula represents a solution to the unification problem, however, maximal and most general unifiers are interpreted as the best solutions to the unification problem. A logic has the {\it unitary} type of unification if for any unifiable formula in logic there is a mgu. If there are unifiable formulas that don't have a mgu, then logic can have the following types of unification: {\it finitary} type, if only finite sets of maximal unifiers exist for each such formula in logic; {\it infinitary} if there are formulas having an infinite number of maximal unifiers; {\it nullary} if some unifiable formulas don't have maximal unifiers, \cite{Ghil2,Jer_K}.

A substitution $\tau$ is \textit{projective unifier} for $\alpha(p_1,\dots,p_s)$ in logic $\mathcal{L}$, if both of the following conditions are met:
\begin{enumerate}
	\item $\tau$ is an unifier for $\alpha$: $\tau(\alpha)\in \mathcal{L}$;

	\item $\alpha$ is a \textit{projective} formula: $\forall p_i\in Var(\alpha): \Box \alpha \rightarrow [p_i \equiv \tau (p_i)]\in \mathcal{L}$.
\end{enumerate}

The importance of the search for projective unifiers is determined, first of all, by the following known

\begin{lemma} 
\cite{Ghil1} \label{Ghil}
if a substitution $\sigma_p$ is projective for $\varphi$ in $\mathcal{L}$, then $\{\sigma_p\}$ is a complete set of unifiers for $\varphi$ (i.e. $\sigma_p$ is an mgu for $\varphi$).
\end{lemma}

\begin{lemma}\label{ground}
Unifiability of an arbitrary formula $\varphi(p_1,\dots, p_s)$ in $\mathcal{L}$ can be efficiently set via substitutions $\sigma(\varphi)$ following kind: $\forall p_i \in Var(\varphi)$, $\sigma(p_i) \in \{\top, \bot\}$.
\end{lemma}

\begin{proof}
The proof repeats argumentation from \cite{Bash2018}, here we describe its scheme.
Let $\varphi (p_1, \ldots, p_s)$ be unifiable in $\mathcal{L}$ and $\delta_1 (q_1,\dots,q_r), \dots, \delta_s (q_1,\dots,q_r)$ is its unifier. Then it is true that 
$$\delta (\varphi):= \varphi(\delta_1(q_1,\dots,q_r), \dots, \delta_s(q_1,\dots,q_r)) \in \mathcal{L}.$$

Any substitution of the variables $q_1,\dots,q_r$ to constants $c_i \in \{\top, \bot\} (i\in[1,r])$ preserves truth values of the formula, because of $\delta (\varphi)\in \mathcal{L}$, so
$$\varphi(gu(p_1), \dots, gu(p_s)) \in \mathcal{L},$$
where $gu(p_i):= \delta_i(c_1,\dots,c_r) \in \{\top, \bot \}$. Then any such $gu(\varphi)$ is a ground unifier of $\varphi$. To prove the existence of such a unifier of an arbitrary formula $\psi(p_1,\dots,p_s)$, it suffices to consider no more than $2^s$ substitution options $\{\top, \bot\}$ instead of variables. If among them there is a option s.t. $\psi(gu (p_1), \dots, gu(p_s)) \equiv_{\mathcal{L}} \top $, then $\psi$ is unifiable in $\mathcal{L}$, $gu$ is its ground unifier. If for all $2^s$ options $gu(\psi) \notin \mathcal{L}$, then $\psi$ doesn't have a ground unifier, and therefore isn't unifiable in $\mathcal{L}$.
\end{proof}

\section{Some obvious corollaries}

In 2012, W. Dzik and P. Wojtylak \cite{DzW} obtained a number of important results for the logic $\mathcal{S}4.3$ and its extensions, the most important of which here are the following statements:

\begin{theorem} (3.18 in \cite{DzW})

Each unifiable formula has a projective unifier in $\mathcal{S}4.3$.
\end{theorem} 

\begin{corollary} (3.19 in \cite{DzW})
A modal logic $\mathcal{L}$ containing $\mathcal{S}4$ has projective unification if and only if $\mathcal{S}4.3\subseteq \mathcal{L}$.
\end{corollary}

The formula $L:= \Box(\Box p \rightarrow q)\vee \Box(\Box q \rightarrow p)$ is deducible in the logic $\mathbf{PM4}$, and is an axiom in $\mathbf{PM1}$, which means $\mathbf{PM1},\mathbf{PM4}$ is a normal extensions of $\mathbf{S}4.3$. The McKinsey's formula $M = \Box\Diamond p \leftrightarrow \Diamond\Box p$ is deducible~in~$\mathbf{PM1}$ and is an axiom of $\mathbf{PM4}$. By virtue of these properties, and as a corollary of Theorem 4.2 and Corollary 3.18 proved in \cite{DzW}, the following is true:

\begin{corollary}
For pretabular logics $\mathbf{PM_{1,4}} \in \{\mathbf{PM1}, \mathbf{PM4}\}$ the following conditions hold and are equivalent:
\begin{enumerate}
	\item $\mathbf{PM_{1,4}}$ --- structurally complete;
	\item $\mathbf{PM_{1,4}}$ --- hereditarily structurally complete;
	\item $\mathbf{S}4.3\mathbf{M} \subseteq \mathbf{PM_{1,4}}$.
\end{enumerate}
\end{corollary}

And by the Lemma~\ref{Ghil}, in conjunction with the above conditions,

\begin{corollary} \label{PM4}
Pretabular logics $\mathbf{PM1}$ and $\mathbf{PM4}$:
\begin{itemize}
	\item have projective unification;
	\item have unitary type of unification.
\end{itemize}
\end{corollary}

In this paper, for the case of logic $\mathbf{PM4}$, we will separately consider the question of projective unification and the type of projective unifier adequate for this logic will be described.

The case of the logic $\mathbf{PM5}$ is completely described by \cite{Dz_S5}, in which W. Dzik proves unitary type of unification for $\mathcal{S}5$ and suggests the form of the most general unifier:

\begin{theorem} (6 in \cite{Dz_S5})
Modal logic $\mathcal{S}5$ (both in the standard formalization and in the formalization with strict implication) and all its extensions have unitary unification.
\end{theorem} 

The substitution $\sigma(x_i):=
\begin{cases}
   \Box \varphi\rightarrow x, &\text{if\hphantom{a}} gu(x)=\top\\
   \Box\varphi\wedge x, &\text{if\hphantom{a}} gu(x)= \bot
 \end{cases}$ is the mgu for every unifiable in $\mathcal{S}5$ formula $\varphi$.

\section{Counterexamples for cases $\mathbf{PM2}$ and $\mathbf{PM3}$}

\subsection{$\bm{\mathbf{PM2}},\bm{\mathbf{PM3}}$ have no unitary type}

As we know from the Lemma \ref{Ghil}, if a logic has projective unification, then any unifiable formula in this logic has a mgu, which means that a logic itself has the unitary type of unification.

\subsubsection{Unification in $\bm{\mathbf{PM2}}$ is not projective}

Lets show that the logic $\mathbf{PM2}$ doesn't have projective unification using the following example.

\begin{example}
$$ \varphi(x):= \Box x \vee \Box\neg x$$
\end{example}

Given formula is unifiable in the logic $\mathbf{PM2}$: its ground unifier is the substitution $gu: x\mapsto \bot$.

Let $\varphi$ be the formula in the language of $\mathbf{PM2}$ and has projective substitution $u$, i.e. by definition $\forall x\in Var(\varphi): \Box\varphi \rightarrow (x\equiv u(x)) \in \mathbf{PM2}$. We show that then $u$ cannot be an unifier for $\varphi$ in $\mathbf{PM2}$. 
To do this, consider the model $M:=\langle \mathbb{V}_m,\upsilon \rangle$, presented in Fig. 1.

\begin{figure}[h] \centering
	\includegraphics[width=6cm]{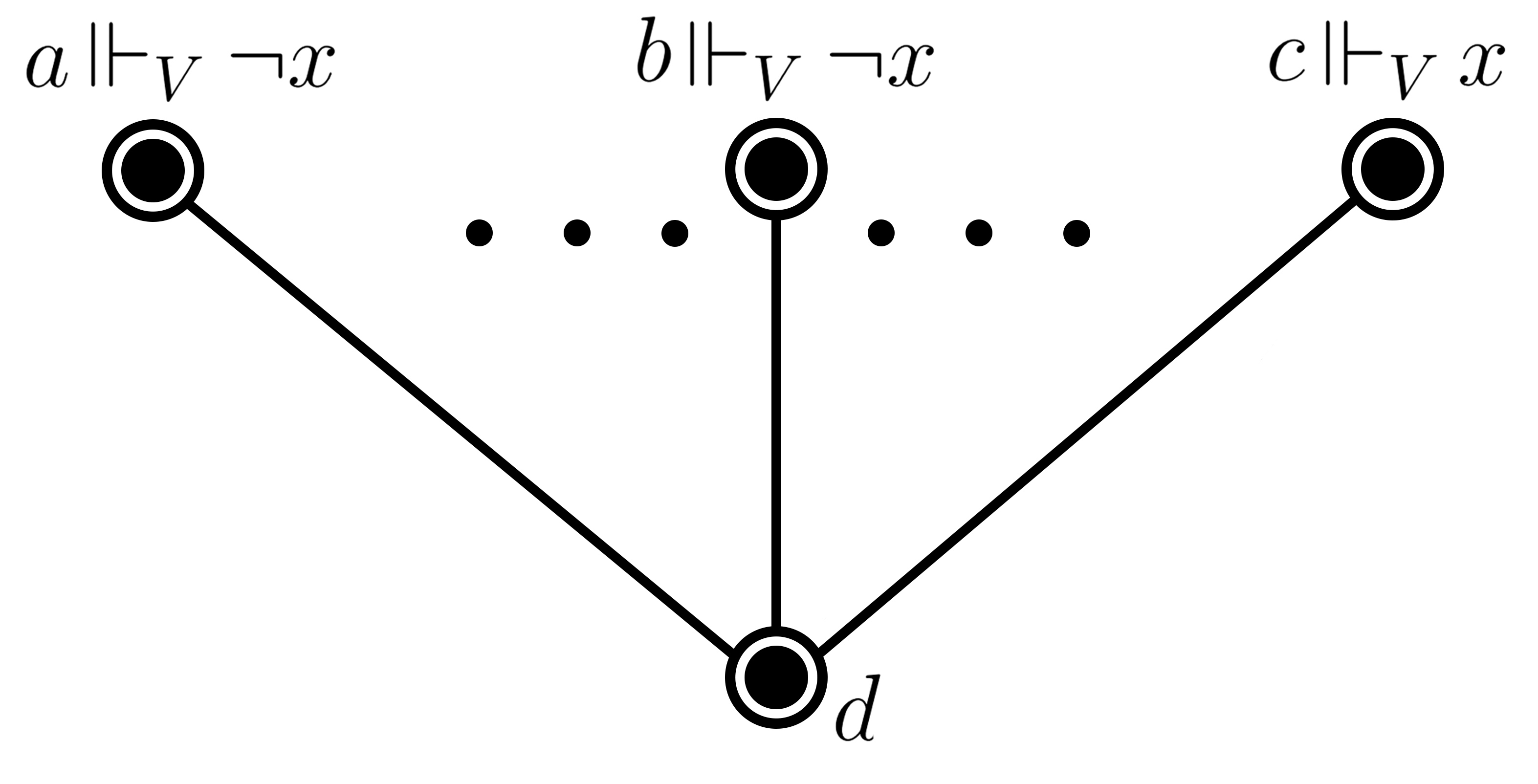}
	\caption{The model $M:=\langle \mathbb{V}_m,\upsilon \rangle$.}
\end{figure}

By virtue of $\langle \mathbb{V}_m,a \rangle \Vdash_{\upsilon} \neg x$, and therefore $\langle \mathbb{V}_m,a \rangle \Vdash_{\upsilon} \Box\neg x$ (because of $a$ is the first layer cluster of the model),
it is true that $\langle \mathbb{V}_m,a \rangle \Vdash_{\upsilon} \varphi$, thus $\langle \mathbb{V}_m,a \rangle \Vdash_{\upsilon} u(\varphi)$, where $u: x\mapsto u(x)$ with the condition $x\equiv u(x)$ (the simplest example of such a substitution is the case $u: x\mapsto x$). 

Similar reasoning is true for the point $b$, 
as well as $c$, at which holds $\langle \mathbb{V}_m,c \rangle \Vdash_{\upsilon} \Box x$, and hence again $\langle \mathbb{V}_m,c \rangle \Vdash_{\upsilon} \varphi$ and $\langle \mathbb{V}_m,c \rangle \Vdash_{\upsilon} u(\varphi)$.

However, regardless of valuation of the variable $x$ at the point $d$ (i.e. $\langle \mathbb{V}_m,d \rangle \Vdash_{\upsilon} x$ or $\langle \mathbb{V}_m,d \rangle \Vdash_{\upsilon} \neg x$),
holds $\langle \mathbb{V}_m,d \rangle \nVdash_{\upsilon} \Box x$ and $\langle \mathbb{V}_m,d \rangle \nVdash_{\upsilon} \Box\neg x$, and therefore $\langle \mathbb{V}_m,d \rangle \nVdash_{\upsilon} \varphi$. In that case, for $u: x\mapsto u(x)$ s.t. $x\equiv_{\mathbf{PM2}}u(x)$, holds $\langle \mathbb{V}_m,d \rangle \nVdash_{\upsilon} u(\varphi)$, so $u$ cannot be an unifier for the formula $\varphi$ in logic.

Thanks to this example, there are unifiable, but not projective formulas in $\mathbf{PM2}$. Therefore, the following is true

\begin{lemma}
The logic $\mathbf{PM2}$ does not have projective unification.
\end{lemma}

As we already noted above, unitary type in a logic follows from projective unification, but not vice versa. At the same time, the absence of unitary type for the logic $\mathbf{PM2}$ follows from the formula already considered in the example above and the fact previously noted by S. Ghilardi \cite{Ghil2}. To prove this, we introduce the concept of projective approximation \cite {Jer_K}, proposed by Ghilardi for describing finite complete sets unifiers for formulas.

\textit{Projective approximation} of a formula $\varphi$ is a finite set $\Pi$ of projective formulas, s.t. $\varphi \vdash_{\mathcal{L}}\Pi$ and $\forall \pi\in\Pi: \pi\vdash_{\mathcal{L}}\varphi$. If $\Pi$ if a projective approximation of $\varphi$, then a set of projective unifiers $\Pi$ defines a finite complete set of unifiers for $\varphi$ in logic. 


The formula $\varphi_1(x):= \Box\neg x$ is unifiable: the substitution $gu_1: x\mapsto \bot$ is its unifier in $\mathbf{PM2}$. Besides, $\Vdash_{\mathbf{PM2}} \Box\neg\bot$ and $\Box\neg x \Vdash_{\mathbf{PM2}} x \leftrightarrow\bot$, and therefore $gu_1$ is the projective unifier and $\varphi_1$ is the projective formula in $\mathbf{PM2}$.

Similarly, the formula $\varphi_2(x):= \Box x$ is unifiable: the substitution $gu_2: x\mapsto \top$ is its unifier in $\mathbf{PM2}$, moreover $\Vdash_{\mathbf{PM2}} \Box\top$ and $\Box x \Vdash_{\mathbf{PM2}} x \leftrightarrow\top$, and therefore $gu_2$ is the projective unifier and $\varphi_2$ is the projective formula in $\mathbf{PM2}$.

The following disjunctive property holds for the formula $\varphi(x):= \Box x \vee \Box\neg x$:

$$\Box B \Vdash_{\mathbf{PM2}} \Box x \vee \Box\neg x \Rightarrow (\Box B \Vdash_{\mathbf{PM2}} \Box x) \vee (\Box B \Vdash_{\mathbf{PM2}} \Box\neg x).$$

It means that $\Pi(\varphi):= \{\Box x, \Box\neg x\}$. Therefore, the formula $\varphi$ has two maximal unifiers and doesn't have a mgu. Consequently, there is an unifiable formulas in $\mathbf{PM2}$, which don't have a mgu, therefore

\begin{lemma}
The logic $\mathbf{PM2}$ does not have unitary type of unification.
\end{lemma}

\subsubsection{Unification in $\bm{\mathbf{PM3}}$ is not projective}

We also show by the example that the logic $\mathbf{PM3}$ doesn't have projective unification.

\begin{example}
$$ \varphi(x_1,x_2)= {L}:= \Box (\Box x_1 \rightarrow x_2) \vee \Box(\Box x_2 \rightarrow x_1),$$
where ${L}$ is the Lemmon's formula for the reflexive case of logic.
\end{example}

Given formula is unifiable in $\mathbf{PM3}$: ground unifiers for it are following substitutions:
$$gu_1: x_1\mapsto \top, x_2\mapsto \alpha;$$
$$gu_2: x_1\mapsto \alpha, x_2\mapsto \top,$$
where $\alpha$ is an arbitrary formula in the language of logic.

Let $\varphi$ be projective in $\mathbf{PM3}$ and $u$ is its projective substitution.
Then, by the definition of a projective formula, $\Box{L} \rightarrow (x\equiv u(x)) \in \mathbf{PM3}$. We show that in this case $u$ cannot be an unifier of ${L}$. 
To do this, consider the model $M:=\langle\mathbb{U} _m,\upsilon\rangle$, shown in Fig.~2.

\begin{figure}[h] \centering
	\includegraphics[width=7cm]{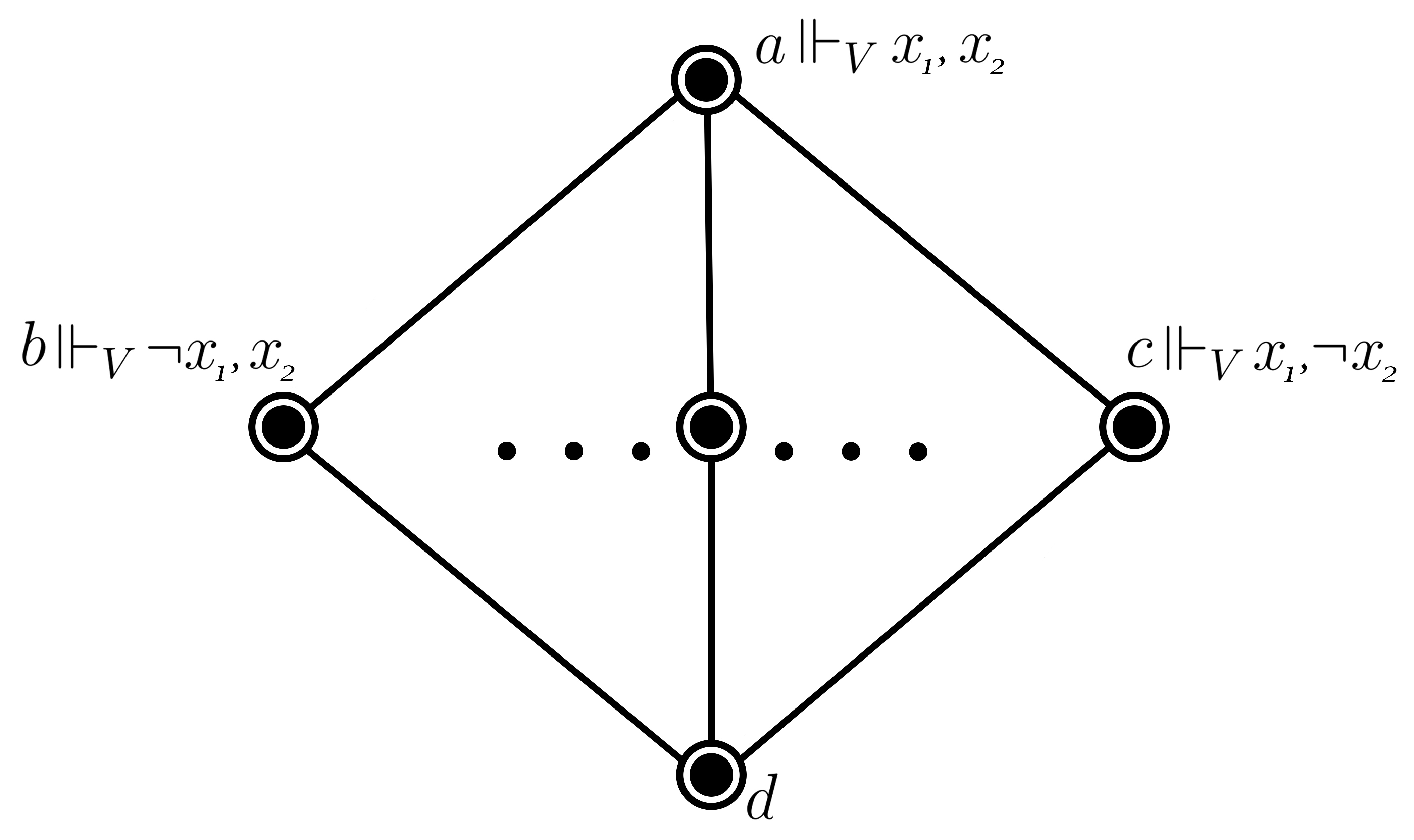}
	\caption{The model $M:=\langle \mathbb{U}_m,\upsilon \rangle$.}
\end{figure}

By virtue of $\langle \mathbb{U}_m,a \rangle \Vdash_{\upsilon} x_1, x_2$, holds $\langle \mathbb{U}_m,a \rangle \Vdash_{\upsilon} \Box (\Box x_1 \rightarrow x_2) \wedge \Box(\Box x_2 \rightarrow x_1)$, which means the more so $\langle \mathbb{U}_m,a \rangle \Vdash_{\upsilon} {L}$ and $\langle \mathbb{U}_m,a \rangle \Vdash_{\upsilon} u(L)$, where $u: x\mapsto u(x)$ with the condition $x\equiv u(x)$.

At the point $b$, by virtue of $\langle \mathbb{U}_m,b \rangle \Vdash_{\upsilon} \neg x_1, x_2$, the first term of the disjunction is valid $\langle \mathbb{U}_m,b \rangle \Vdash_{\upsilon} \Box x_1 \rightarrow x_2$. At the incomparable point $c$ of the same layer --- the second term is valid, i.e. $\langle \mathbb{U}_m,c \rangle \Vdash_{\upsilon} \Box x_2 \rightarrow x_1$, hence $\langle \mathbb{U}_m,\{b,c\} \rangle \Vdash_{\upsilon} {L}$ and again $\langle \mathbb{U}_m,\{b,c\} \rangle \Vdash_{\upsilon} u(L)$.

However, regardless of truth values of the variables $ x_1, x_2 $ at the point $d$, it is true that $\langle \mathbb{U}_m,d \rangle \nVdash_{\upsilon} \Box (\Box \sigma(x_1) \rightarrow \sigma(x_2))$ (because $\langle \mathbb{U}_m,c \rangle \nVdash_{\upsilon} \Box \sigma(x_1) \rightarrow \sigma(x_2)$) and $\langle \mathbb{U}_m,d \rangle \nVdash_{\upsilon} \Box (\Box \sigma(x_2) \rightarrow \sigma(x_1))$ ($\langle \mathbb{U}_m,b \rangle \nVdash_{\upsilon} \Box \sigma(x_2) \rightarrow \sigma(x_1)$). It follows that $\langle \mathbb{U}_m,d \rangle \nVdash_{\upsilon} L$, therefore for $u: x\mapsto u(x)$, where $x\equiv_{\mathbf{PM2}}u(x)$, holds $\langle \mathbb{U}_m,d \rangle \nVdash_{\upsilon} u(L)$, which means that $u$ cannot be an unifier for the formula $L$ in logic.

As in the case of $\mathbf{PM2}$, in $\mathbf{PM3}$ non-projective unifiable formulas also hold and is true:

\begin{lemma}
The logic $\mathbf{PM3}$ does not have projective unification.
\end{lemma}

\section{${\bm{\mathbf{PM2}}}$, ${\bm{\mathbf{PM3}}}$ have finitary type}

S. Ghilardi in \cite{Ghil2} noted that many well-known systems (such as $\mathcal{K}4$,$\mathcal{S}4$, $\mathcal{S}4\mathcal{G}rz$, $\mathcal{GL}$ et al.) have a finitary type. Further, we prove that this is also true for the cases under consideration $\mathbf{PM2}$ and $\mathbf{PM3}$.

V. Rybakov proposed an approach to the description of finite complete sets of unifiers \cite{R2011cu}, based on the modified technique previously used in working with the problem of admissibility \cite{R1984}. However, that approach was based on the technique applied to the modal system $\mathbf{S}4$, which has the property of co-covering, which is not applicable in the case of the pretabular systems studied in this paper, because of the limited depth of the scales. At the same time, the pretabular logic $\mathbf{PM2}$ and $\mathbf{PM3}$ possess the weak co-covering property \cite{RK2013}.

The logic $\mathcal{L}\in Ext(\mathcal{S}4)$ has \textit{weak co-covering property}, if for any finite root $\mathcal{L}$-scale $F = b^{R}$ and arbitrary antichain of clusters $\mathcal{X}$ (i.e. incomparable in R) from $F\setminus b$, adding a reflexive cluster $c$ as a co-covering of the anti-chain to scale $\bigcup_{c\in \mathcal{X}^R} c^R$ also gives a $\mathcal{L}$-scale.

In this section, we give a modification of Rybakov's technique proposed for modal logic of finite layers \cite{R1984FL}, which is in good agreement with our case. Let us construct a description of complete sets of unifiers for an arbitrary unifiable formula in the logics $\mathbf{PM2}$ and $\mathbf{PM3}$ and show that both logics have a finitary type of unification, i.e. all complete sets are finite. To do this, at the beginning we define the $n$-characteristic model for these logics, the reduced normal form (rnf) of the formula, and also construct $n$-characteristic models of a special kind on sub-formulas of rnf with special valuations, which, as will be shown below, will give a unifier for an arbitrary unifiable formula in the corresponding logic.

\subsection{$\bm{\mathbf{PM2}}$, $\bm{\mathbf{PM3}}$: $\bm{n}$-characteristic models}

A model $M$ with a valuation $\upsilon$ of variables $p_1,\dots,p_n$ called \textit{$n$-characterization} for modal logic $\mathcal{L}$, if $\forall \varphi(p_1,\dots,p_n)$: $\varphi\in\mathcal{L} \Leftrightarrow M\Vdash_V\varphi$.

Recall that \textit{cluster} $C\subseteq W$ of a scale $F:= \langle W,R \rangle$ is a set of points (or elements) from $W$, such that
\begin{enumerate}
	\item $\forall a,b \in C: a R b;$
	\item $\forall a \in C, \forall b \in W : (a R b \& b R a) \rightarrow (b \in C).$
\end{enumerate}

\textit{Depth} of a cluster $x$ from a scale $F$ is the maximum number of clusters in chains starting with a cluster containing $x$. Denote:

\begin{itemize}
	\item $S_i(F)$ is a set of clusters depth $i$ from $F$;
	\item $Sl_i(F)$ is a set of clusters depth no more than $i$ from $F$;
	\item $M^{'}$ is a sub-model of a model $M$, i.e. subset $M^{'}\subseteq M$ with a relations and valuation induced from $M$. A sub-model $M^{'}$ is \textit{open}, if its clusters include all $R$-accessible from $M$;
	\item $u^n_i$ is a projection function (selecting the $i$th argument from $n$); for the case of $n$-characterization, we will omit the superscript.
\end{itemize}

Let $\mathbb{V}_m:= \langle V_m,R \rangle$ be the scale of logic $\mathbf{PM2}$, and $\mathbb{U}_m:= \langle U_m,R \rangle$ be the scale of logic $\mathbf{PM3}$ (all clusters of any models $\mathbb{V}_m$ and $\mathbb{U}_m$ of logics $\mathbf{PM2}$ and $\mathbf{PM3}$ are singular), $n$ is a natural number. Fix the variables $p_1, \dots, p_n$ and its valuation $V$.

Let $\mathcal{T}_{n}^1 := \langle T_{n}^1, R_{n}^1, \upsilon_{n}^1 \rangle$, where $T_{n}^1 = \bigcup_{i\in \{1,\dots,n\}} p_i$, $p_i$ form incomparable clusters by $R_{n}^1$, $\upsilon_n^1(p_i) := \{p_i | p_i\in p_i\}$.

Suppose the model $\mathcal{T}_n^{l}$ already built. Put the set of all anti-chains of clusters $W_l$ (i.e. incomparable in $R_n^l$ clusters) of model $\mathcal{T}_{n}^l$, containing at least one cluster from $S_l(\mathcal{T}_{n}^l)$. Then  $W_{l+1}$ define as follows:
$$Sl_{l+1}:= (W_l\times C) \backslash D,$$
where $C:= \{\langle p_i, l+1\rangle | i\in \{1,\dots,n\}\}$ is new clusters of the model depth $l+1$, $D:= \{\langle L, \langle p_i, l+1 \rangle \rangle | L\in W_l, L= \{\nabla\},\, \text{cluster}  \,\, \langle p_i, l+1\rangle \,\, \text{with the valuation} \,\, \langle p_i, l+1\rangle\Vdash_\upsilon p_i \Leftrightarrow p_i\in p_i \,\, \text{isomorphic to submodel of cluster} \,\, \nabla\}$.

Let $T_n^{l+1}:= T_n^l \cup Sl_{l+1},$ and $\forall x\in Sl_{l+1}$:
$$x \overline{R^{l+1}} y \Leftrightarrow [x=y]\vee [\exists z \exists t (z\in u_1(x)\wedge t\in z\wedge t R_n^{l} y)] \vee$$
$$\vee [(y\in Sl_{l+1}) \wedge (u_1(x)= u_1(y))\wedge(u_2(u_2(x))=u_2(u_2(y)))],$$
$$R_n^{l+1}:=\overline{R^{l+1}}\cup R_n^{l},$$
$$\upsilon_n^{l+1}(p_i):= \upsilon_n^l (p_i)\wedge \{x | x\in Sl_{l+1}, p_i \in u_1(u_2(x))\},$$
$$\mathcal{T}_{n}^{l+1}:= \langle T_{n}^{l+1}, R_{n}^{l+1}, \upsilon_{n}^{l+1} \rangle.$$

Each model $\mathcal{T}_{n}^j$ constructed in this way is an open sub-model of model of the next layer --- $\mathcal{T}_{n}^{j+1}$ --- and consists of its clusters of depth no more than $j$.
Due to the finite model property of $\mathbf{PM2}$, $\mathbf{PM3}$, the method of constructing $n$-characteristic models from \cite{R1984FL}, as well as reasoning from \cite{Rym96} for modal extensions of $\mathbf{S}4$ depth 2, holds

\begin{lemma}\label{ncharPM2}
The model $\mathcal{T}_{n}^{2}:= \langle T_{n}^{2}, R_{n}^{2}, \upsilon_{n}^{2} \rangle$ is an $n$-characterization for $\mathbf{PM2}$:
$$ \forall \varphi(p_1,p_2,\dots,p_n): \varphi\in\mathbf{PM2} \Leftrightarrow \mathcal{T}_{n}^{2}\Vdash \varphi.$$
\end{lemma}






\begin{lemma}
The model $\mathcal{T}_{n}^{3}:= \langle T_{n}^{3}, R_{n}^{3}, \upsilon_{n}^{3} \rangle$ is an $n$-characterization for $\mathbf{PM3}$:
$$ \forall \varphi(p_1,p_2,\dots,p_n): \varphi\in\mathbf{PM3} \Leftrightarrow \mathcal{T}_{n}^{3}\Vdash \varphi.$$
\end{lemma}

A set of clusters $W'\subseteq W$ called \textit{definable} in $M:= \langle F,\upsilon \rangle$, if there is a formula $\varphi(p_1,p_2,\dots,p_n)$, s.t. $\upsilon(\varphi(p_1,p_2,\dots,p_n)) = W'$.

\begin{lemma}
Any cluster of of the model $\mathcal{T}_{n}^{2}$ and $\mathcal{T}_{n}^{3}$ is definable.
\end{lemma}

This lemma can be reformulated as follows:

\textit{For any cluster $w\in \mathcal{T}_{n}^{j}$ there is $\varphi(p_1,p_2\dots,p_n)$, s.t. $\forall s\in \mathcal{T}_{n}^{j}$:}
$$ \langle \mathcal{T}_{n}^{j},s\rangle \Vdash \varphi \Leftrightarrow s=w.$$

The proof of the lemma is a special case of Lemma 2 from \cite{R1984FL}.

\subsection{Reduced normal form.}

For the following construction of Kripke models of a special form, we define the \textit{reduced normal form} (\textit{rnf}, for short) of formulas and show that to solve the unification problem it is sufficient to consider only formulas given in the rnf.

By the scale $\mathbb{F}_m$ in this and the next section we mean cases of scales $\mathbb{V}_m$ (for $\mathbf{PM2}$) or $\mathbb{U}_m$ (for $\mathbf{PM3}$).
 
A formula $\varphi(x_1,\dots,x_n)$ called \textit{given in rnf}, if:

\begin{equation}\label{rnf}
\varphi := \bigvee_{1 \leq j \leq k} (\bigwedge_{1\leq i \leq n} [x_i^{t(j,i,0)}\wedge (\Diamond x_i)^{t(j,i,1)}]),
\end{equation}
where $x_i$ is a variables of $\varphi$, $t(j,i,z)\in\{0,1\}$ and $\forall \alpha := \{x_i, \Diamond x_i \}$:

\begin{equation*}
\alpha^t := 
\begin{cases}
\alpha &\text{if $t=0$};\\
\bar{\alpha} &\text{if $t=0$.}
\end{cases}
\end{equation*}

A formula $\varphi_{rf}$ called \textit{rnf} for $\varphi$, if the following conditions hold:
\begin{enumerate}
	\item $\varphi$ has the form (\ref{rnf});
	\item $Var(\varphi) \subseteq Var(\varphi_{rf})$;
	\item For any scale $\mathbb{F}_m:=\langle W_m,R \rangle$, cluster $a\in W_m$ and valuation $\upsilon$ of variables of $\varphi$ on the scale $\mathbb{F}_m$, if $\langle \mathbb{F}_m,a \rangle \Vdash_{\upsilon} \varphi$, then there is such an extension $\upsilon_1$ of the valuation $\upsilon$ by additional variables of $\varphi_{rf}$, s.t. $\langle \mathbb{F}_m,a \rangle\Vdash_{\upsilon_1} \varphi_{rf}$;
	\item For any scale $\mathbb{F}_m:=\langle W_m,R \rangle$, cluster $a\in W_m$ and valuation $\upsilon_1$ of variables of $\varphi_{rf}$ on the scale $\mathbb{F_m}$ if $\langle \mathbb{F}_m,a \rangle \Vdash_{\upsilon_1} \varphi_{rf}$, then $\forall b\in W_m$ $\langle \mathbb{F}_m,a \rangle \Vdash_{\upsilon} \varphi$, where $\upsilon$ is the limitation of valuation $\upsilon_1$ to variables of $\varphi$.
\end{enumerate}

The following theorem is known:

\begin{theorem} \label{RNF}\cite{Rbook}
There is an algorithm that for exponential time for any given formula $\varphi$ builds its rnf $\varphi_{rf}$.
\end{theorem}

By virtue of the theorem \ref{RNF} the following is true

\begin{lemma}\label{unif_rf}
For the logic $\mathcal{L}\in \{\mathbf{PM2},\mathbf{PM3}\}$ the following statements are true:

\begin{enumerate}
	\item $\varphi$ is unifiable in $\mathcal{L}  \Leftrightarrow  \varphi_{rf}$ is unifiable in $\mathcal{L}$;
	\item if $\sigma$ is an unifier for $\varphi$ in $\mathcal{L}$, then some extension $\sigma$ to additional variables of $\varphi_{rf}$ is an unifier for $\varphi_{rf}$ in $\mathcal{L}$;	\item if $\sigma$ is an unifier for $\varphi_{rf}$ in $\mathcal{L}$, then limitation $\sigma$ to the variables of formula $\varphi$ is an unifier for $\varphi$ in $\mathcal{L}$.
\end{enumerate}
	
\end{lemma}

This lemma is a special case of Lemma 4 of \cite{R2011cu}, by virtue of the Kripke completeness of logics $\mathbf{PM2},\mathbf{PM3}$.

By Lemma 
9, it is sufficient to consider rnf of the formulas instead of themselves for the study of the unification problem, if it will be convenient. Let us further consider special Kripke models constructed using rnf.

Let $\varphi$ be a formula and 
$\varphi_{rf}:=\bigvee_{j\in J} \varphi_j$ be its rnf, 
in accordance with the Theorem \ref{RNF}, 
where $\varphi_j := \bigwedge_{1\leq i \leq n} [x_i^{t(j,i,0)}\wedge (\Diamond x_i)^{t(j,i,1)}]$. 
Let $M(\varphi_{rf}):=\langle W, R,\upsilon \rangle$ be a model, 
built on the basic set $W:=\{\varphi_j\}_{j\in J}$ 
with a valuation $\upsilon$ of all variables of $\psi$, 
given as follows: $\upsilon(x_i)$ is a set of all  
$\varphi_j \in W$, in which $x_i$, 
as free by $\Diamond$ variable, only positive. 
We define the accessibility relation $R$ on the model $M(\varphi_{rf})$, 
using two novel sets:
\begin{itemize}
	\item $\forall \varphi_j: \Theta_1(\varphi_j)$ 
	is a set of all variables, that occur in $\varphi_j$ 
	without negation and free from the operator $\Diamond$;
	\item $\forall \varphi_j: \Theta_2(\varphi_j)$ 
	is a set of all variables, that occur in $\varphi_j$ 
	without negation and bound by the operator $\Diamond$ 
	(i.e. $\Diamond x$).
\end{itemize}

Then 
$\forall M(\varphi_{rf}): \varphi_i R \varphi_j \Leftrightarrow \Theta_2(\varphi_j)\subseteq \Theta_2(\varphi_i)$.

\subsection{Complete sets of unifiers in $\bm{\mathbf{PM2},\mathbf{PM3}}$.}

The arguments proposed below for constructing special kind models are correct for the case of both logics under consideration, taking into account the fact that the corresponding initial models are $\mathcal{T}_{n}^{2}$ for $\mathbf{PM2}$ or $\mathcal{T}_{n}^{3}$ for $\mathbf{PM3}$. Therefore, in this section we give reasoning only for the case of $\mathbf{PM2}$, for $\mathbf{PM3}$ they can be obtained exactly the same way.

\begin{lemma}
Let $\varphi_{rf}:= \bigvee_{j\in J} \varphi_j$ be a formula, given in rnf. If $\varphi_{rf}$ is unifiable in $\mathbf{PM2}$, 
and $\sigma$ is its unifier, 
then there is a model $M(\varphi_{rf})$ on some subformulas $\varphi_j$ 
from a disjunction as clusters of a basic set
$W:=\{\varphi_j\}_{j\in J}$, s.t. 
$\bigvee_{\varphi_j\in M(\varphi)} \sigma(\varphi_j) \in \mathbf{PM2}$ 
and the following is true:
\begin{enumerate}
	\item $\Theta_1(\varphi_j) \subseteq \Theta_2(\varphi_j), \varphi_j\in M(\varphi);$
	\item $\forall \varphi_j\in W, \varphi_j\Vdash_V \varphi_j; $
	\item $\forall D\subseteq W, \exists\varphi_{D,j}\in W:$
	$$ \Theta_2(\varphi_{D,j}) = [\Theta_1(\varphi_{D,j}) \cup \bigcup_{\varphi_j\in D} \Theta_2(\varphi_j)]. $$
\end{enumerate}
\end{lemma}

\begin{proof}
Because the formula $\varphi(p_1,\dots, p_{n})$ is unifiable, 
$\sigma$ is its unifier, 
there is a valuation $S$ of variables $p_1,\dots, p_{n}$ on the scale of model $\mathcal{T}_{r}^{2}$, constructed in accordance with the Lemma \ref{ncharPM2} for finite $r = |Var(\sigma(\varphi))|$, s.t. 
$\langle \mathcal{T}_{r}^{2} \rangle\Vdash_{S} \varphi$. Let $X$ be a set of all $\varphi_j$, 
taken from a disjunction, s.t. $S(\varphi_j)\neq \varnothing$.

(1.)
It's obvious that
$\forall \varphi_j \in X$: $\Theta_1(\varphi_j) \subseteq \Theta_2(\varphi_j)$, 
due to the reflexivity of the model $\mathcal{T}_{r}^{2}$.

(2.) Note that in 
$\varphi_j := \bigwedge_{1\leq i \leq n} [{p_i}^{t(j,i,0)}\wedge (\Diamond {p_i})^{t(j,i,1)}]$
the first --- non-modal --- part of the conjunction $\varphi_j$ is valid,
i.e. $\forall \varphi_j\in W, \varphi_j \Vdash_S {p_i}^{t(j,i,0)}$. 
Let $\Diamond p_i$ be the conjunction member of $\varphi_j$. 
By definition of $X$, $\exists a \in S(\varphi_j)$ 
and some cluster $b\in W$, s.t. $a R_{r}^2 b$, 
$a \Vdash_S \Diamond p_i$ and $b \Vdash_S p_i$. 
In that case, $b\in S(\varphi_k)$ for some $k$, 
which means $p_i \in \Theta_1(\varphi_k)$ 
and $\varphi_k \Vdash_S p_i$. 
We show that $\varphi_j R_{r}^2 \varphi_k$. 
If $p_t \in \Theta_2(\varphi_j)$, 
then $b \Vdash_S p_t$, 
that means $a \Vdash_S \Diamond p_t$ 
and $p_t \in S(\varphi_k)$.

Hence, $\varphi_j R_{r}^2 \varphi_k$ 
and  $\varphi_j \Vdash_S \Diamond p_i$. 
Conversely, let $\varphi_j \Vdash_S \Diamond p_i$, 
then there is $\varphi_k \in X$, s.t.
$\varphi_j R_{r}^2 \varphi_k$ 
and $\varphi_k \Vdash_S p_i$. 
Then $p_i \in \Theta_1(\varphi_k)$ 
and $ p_i \in \Theta_2(\varphi_j)$, 
i.e. $\Diamond  p_i$ is the conjunction member of $\varphi_j$.

(3.) Let $D\subseteq X$, and $e_j\in S(\varphi_j)$ 
be fixed representative for each $\varphi_j\in D$. 
Then we put $e(D):=\{e_j | e_j\in S(\varphi_j), \varphi_j\in D\}$. 
Due to the finiteness of the model $\mathcal{T}_{r}^{2}$, 
the set $e(D)$ is also finite. 
There is such an element $\varphi_{D,j}\in W$, that 

$$ \varphi_{D,j}^{\leq}:= \{b | b\in {\mathcal{T}_{r}^{2}},  \varphi_{D,j} R_{r}^2 b \} = \{ \varphi_{D,j}\} \cup\bigcup \{e_j^{\leq} | e_j \in e(D) \}.
$$

Then $\varphi_{D,j}\in S(\varphi_l)$ for some $\varphi_l\in X$, 
which has all the necessary properties $\varphi_{D,j}$.

\end{proof}

Suppose that for an arbitrary unifiable formula $\varphi_{rf}:= \bigvee_{j\in J} \varphi_j$, $SM(\varphi_{rf},\mathbb{V}_m)$ is a set of all models on sub-formulas $\varphi_j$ as a basic set, satisfying conditions (1)--(3). For any model $M\in SM(\varphi,\mathbb{V}_m)$ and formulas $\varphi_j\in M$ define:

\begin{equation*}
\gamma_j := \varphi_j \wedge (\Box\bigvee_{\varphi_j R \varphi_k, \varphi_k\in M} \varphi_k); \hspace{1cm}
\gamma(M):=\bigvee_{\varphi_j\in M} \gamma_j.
\end{equation*}

Consider the model $Ch(\mathbf{PM2})_k$, where $k=|Var(\varphi)|$,
with valuation $\upsilon$, as well as a sub-model $\upsilon(\gamma(M))$ on those clusters of the model $Ch(\mathbf{PM2})_k$, where $\gamma(M)$ is true wrt the valuation $\upsilon$. Then $\varphi$ is true on the constructed sub-model $\upsilon(\gamma(M))$ wrt valuation $\upsilon$. Define the valuation $\Upsilon$ as follows:

\begin{equation*}
\Upsilon(x_i):= \gamma(M) \wedge \bigvee_{\varphi_j\in M} \varphi_j.
\end{equation*}

Defined valuation $\Upsilon$ for variables of the formula $\varphi$ coincides with valuation $\upsilon$ on the model $\upsilon(\gamma(M))$ and in particular, $\varphi$ is true on $\upsilon(\gamma(M))$ wrt valuation $\Upsilon$.

Using the effective technique proposed in the proof of Lemma 3.4.10 on pages 324--325 in \cite{Rbook}, we extend $\Upsilon$ to a definable valuation $\Upsilon_M$, given on the whole model $Ch(\mathbf{PM2})_k$, where $\varphi$ valid wrt $\Upsilon_M$ at every cluster of $Ch(\mathbf{PM2})_k$ (this is possible due to the execution of properties (1)--(3) above). Then $\varphi$ valid on the model $Ch(\mathbf{PM2})_k$ wrt the resulting valuation $\Upsilon_M$, so $\Upsilon_M$ gives a unifier for $\varphi$ in $\mathbf{PM2}$, and since extension of $\Upsilon$ to $\Upsilon_M$ does not change the truth values on $\upsilon(\gamma(M))$, then $\Upsilon_M$ matches on $\upsilon(\gamma(M))$ with $\Upsilon$. Thus, holds 

\begin{lemma}\label{unif_Sm}
$\forall M\in SM(\varphi,\mathbb{V}_m)$ substitution $\sigma_M$, defining $\Upsilon_M$, 
is an unifier for $\varphi$ and valuation $\Upsilon_M$ matches on $\upsilon(\gamma(M))$ with $\Upsilon$.
\end{lemma}

\begin{lemma}\label{max_unif}
For any unifier $\sigma$ of $\varphi$ there is a model $M\in SM(\varphi,\mathbb{V}_m)$ and substitution $\sigma_1$, s.t. $\sigma(x_i)\equiv_{\mathbf{PM2}}\sigma_1(\sigma_M(x_i))$, where $\sigma_M$ is a substitution, defining valuation $\Upsilon_M$.
\end{lemma}

\begin{proof}
We show that $\sigma(x_i)\equiv_{\mathbf{PM2}}\sigma_1(\sigma_M(x_i))$ for arbitrary unifier $\sigma$. 
Consider the model $Ch(\mathbf{PM2})_k$, where $k$ is a number of all variables 
that occur in all of $\sigma(x_i)$ (i.e. a number of all variables of $\sigma(\varphi)$), $a\in Ch(\mathbf{PM2})_k$.

Due to the fact that $\sigma$ is an unifier of $\varphi$, as shown above there is a model $M\in SM(\varphi,\mathbb{V}_m)$, 
built on some $\varphi_j$ from the disjunction with the properties (1)--(3) and, 
in particular, 
$$Ch(\mathbf{PM2})_k \Vdash_{\upsilon} \bigvee_{\varphi_j\in M(\varphi)} \sigma(\varphi_j).$$
Then holds
$Ch(\mathbf{PM2})_k \Vdash_{\upsilon} \bigvee_{\varphi_j\in M(\varphi)} \sigma(\gamma_j)$ 
and $Ch(\mathbf{PM2})_k \Vdash_{\upsilon} \sigma(\gamma(M)).$

Thus, $\upsilon(\sigma(\gamma (M))) = Ch(\mathbf{PM2})_k$. 
Moreover, by the Lemma \ref{unif_Sm}, 

\begin{equation*}
Ch(\mathbf{PM2})_k \Vdash_{\upsilon} \sigma(\gamma (M)) \rightarrow [\sigma(x_i)\equiv \sigma(\sigma_M(x_i))].
\end{equation*}

In particular, for any variable $x_i$ and $a\in Ch(\mathbf{PM2})_k$

$$\langle Ch(\mathbf{PM2})_k,a \rangle \Vdash_{\upsilon} \sigma(x_i) \Leftrightarrow \langle Ch(\mathbf{PM2})_k,a \rangle \Vdash_{\upsilon} \sigma(\sigma_M(x_i)).$$
\end{proof}

Thus, all constructed unifiers $\sigma_M$ defined by the $\Upsilon_M$ give a finite complete set of unifiers for the formula $\varphi$ in the logic $\mathbf{PM2}$. Using similar reasoning, but taking $\mathcal{T}_{r}^{3}$ as the basic $n$-characteristic model, $\Upsilon_M$ also allows us to construct a finite $CU$ for $\mathbf{PM3}$ too. Therefore, the following is true

\begin{theorem}
Logics $\mathbf{PM2}$ and $\mathbf{PM3}$ have finitary type of unification.
\end{theorem}

\section{${\bm{\mathbf{PM4}}}$ has a unitary type.}

\begin{theorem}
Any unifiable in $\mathbf{PM4}$ formula is projective.
\end{theorem}

\begin{proof}
Let the formula $\varphi (p_1, \dots, p_s)$ be unifiable in $\mathbf{PM4}$.  
For all variables $p_i \in Var(\varphi)$ consider the following substitution $\sigma(p_i)$:
$$\sigma(p_i):= (\Box \varphi\wedge p_i) \vee 
(\Diamond\neg\varphi\wedge gu(p_i)),$$
where $gu(p_1), \dots, gu(p_s)$ is a ground unifier for $\varphi$, obtained by the algorithm proposed in the Theorem \ref{ground}.

Let $M_{\mathbb{Y}_n}:=\langle \mathbb{Y}_n,V\rangle$ be a model of $\mathbf{PM4}$ with arbitrary valuation $V$. 
If $\sigma$ is an unifier for $\varphi$, then $\sigma(\varphi) \in \mathbf{PM4}$ and $\forall x \in\mathbb{Y}_n$ $\langle M_{\mathbb{Y}_n},x\rangle \Vdash_V \sigma (\varphi)$.
We show that the substitution $\sigma$ is a projective unifier of $\varphi$ in $\mathbf{PM4}$. To do this, check both points of the definition of projective unifier.
\begin{enumerate}
	\item $\sigma(\varphi)\in \mathbf{PM4}$. The following cases are possible:
	\begin{enumerate}
		\item If for some $x\in \mathbb{Y}_n: \langle M_{\mathbb{Y}_n},x\rangle \Vdash_{V} \varphi$ and $\forall y\in \mathbb{Y}_n$, s.t. $xRy$, also holds $\langle \mathbb{Y}_n,y\rangle \Vdash_{V} \varphi$, then $\langle \mathbb{Y}_n,x\rangle \Vdash_{V} \Box \varphi$ and, consequently, second disjunctive term $\sigma(p_i)$ is disproved on the cluster $x$. In this case, if $\langle \mathbb{Y}_n,x\rangle \Vdash_{V} p_i$, then $\langle \mathbb{Y}_n,x\rangle \Vdash_{V} \Box \varphi \wedge p_i$ and therefore $\langle \mathbb{Y}_n,x\rangle \Vdash_{V} \sigma(p_i)$. If $\langle \mathbb{Y}_n,x\rangle \Vdash_{V} \neg p_i$, then $\langle \mathbb{Y}_n,x\rangle \nVdash_{V} \Box \varphi \wedge p_i$ and, consequently, $\langle \mathbb{Y}_n,x\rangle \Vdash_{V} \neg\sigma(p_i)$.  Hence we conclude that truth values of $\varphi (p_1, \ldots, p_s)$ on $x$ wrt valuation $V$ coincide with truth values of $\varphi (\sigma(p_1), \ldots, \sigma(p_s))$ at the same cluster wrt $V$, so in this case $\langle \mathbb{Y}_n,x\rangle \Vdash_{V} \sigma(\varphi)$.
		\item If, regardless of $\langle \mathbb{Y}_n,x\rangle \Vdash_{V} \varphi$ or $\langle \mathbb{Y}_n,x\rangle \Vdash_{V} \neg \varphi$, there is a cluster $y\in \mathbb{Y}_n: xRy$, such that on it holds $\langle \mathbb{Y}_n,y\rangle \Vdash_{V} \neg \varphi$, then obviously $\langle \mathbb{Y}_n,x\rangle \Vdash_{V} \Diamond\neg \varphi$, and truth values of any $\sigma(p_i)$ on $x$ coincide with $gu(p_i)$. By virtue of the choice of the ground unifier $gu(\varphi) \in \mathbf{PM4}$, holds $\langle \mathbb{Y}_n,x\rangle \Vdash_{V} gu(\varphi)$, and therefore in this considered case also $\langle \mathbb{Y}_n,x\rangle \Vdash_{V} \sigma(\varphi)$.
	\end{enumerate}

Since all possible variants of valuations are described by this two cases, $\sigma(\varphi)\in \mathbf{PM4}$ (i.e. $\sigma$ is an unifier) for an arbitrary unifiable in $\mathbf{PM4}$ formula $\varphi$. 
	\item $\Box \varphi \rightarrow [p_i \equiv \sigma (p_i)]\in \mathbf{PM4}$ for any variable $p_i\in Var(\varphi)$

When substituting $\sigma(p_i)$ into the expression above (i.e. to the second condition of definition), we get the following: $\forall p_i \in Var(\varphi)$
$$
\Box \varphi \rightarrow (p_i \leftrightarrow [(\Box \varphi\wedge p_i) \vee (\Diamond\neg\varphi\wedge gu(p_i))])\in \mathbf{PM4},  
$$
if $\sigma$ is a projective unifier for $\varphi$. By contradiction: let $\sigma$ does not satisfy 2nd condition. Then $\exists x\in \mathbb{Y}_n$
\begin{equation}
\langle \mathbb{Y}_n,x\rangle \Vdash_{V} \Box \varphi, 
\end{equation}
but
\begin{equation}
\langle \mathbb{Y}_n,x\rangle \nVdash_{V} p_i \leftrightarrow [(\Box \varphi\wedge p_i) \vee (\Diamond\neg\varphi\wedge gu(p_i))].
\end{equation}

In this case,
\begin{equation}
\langle \mathbb{Y}_n,x\rangle \nVdash_{V} p_i \rightarrow [(\Box \varphi\wedge p_i) \vee (\Diamond\neg\varphi\wedge gu(p_i))], 
\end{equation}

or
\begin{equation}
\langle \mathbb{Y}_n,x\rangle \nVdash_{V} [(\Box \varphi\wedge p_i) \vee (\Diamond\neg\varphi\wedge gu(p_i))] \rightarrow p_i. 
\end{equation}

If there is (3), holds $\langle \mathbb{Y}_n,x\rangle \Vdash_{V} p_i$, but, by virtue of validity of (1) and $p_i$ on $x$, $\langle \mathbb{Y}_n,x\rangle \Vdash_{V} \Box \varphi\wedge p_i$, and therefore $\langle \mathbb{Y}_n,x\rangle \Vdash_{V} p_i \rightarrow [(\Box \varphi\wedge p_i) \vee (\Diamond\neg\varphi\wedge gu(p_i))]$.

If there is (4), holds $\langle \mathbb{Y}_n,x\rangle \Vdash_{V} [(\Box \varphi\wedge p_i) \vee (\Diamond\neg\varphi\wedge gu(p_i))]$. This is only possible with $\langle \mathbb{Y}_n,x\rangle \Vdash_{V} p_i$, because $\langle \mathbb{Y}_n,x\rangle \Vdash_{V} \Box \varphi$ by virtue of (1), which means that in $\sigma (p_i)$ Only the first disjunct can be valid. Therefore, the conclusion of the formula (4) is true and $\langle \mathbb{Y}_n,x\rangle \Vdash_{V} [(\Box \varphi\wedge p_i) \vee (\Diamond\neg\varphi\wedge gu(p_i))] \rightarrow p_i.$
From everything described above it follows that $\sigma$ is the projective unifier for 
$\varphi$ in $\mathbf{PM4}$, which means $\varphi$ itself is projective.
\end{enumerate}
\end{proof}

By virtue of what has been proved, for any unifiable in $\mathbf{PM4}$ formula $\varphi$ there is a projective unifier, the construction of which is proposed in the proof scheme. By Lemma \cite{Ghil1}, holds

\begin{corollary}
Let  $\varphi (p_1, \dots, p_s)$ be an arbitrary unifiable formula in $\mathbf{PM4}$, and $\sigma(p_i):= (\Box \varphi\wedge p_i) \vee 
(\Diamond\neg\varphi\wedge gu(p_i))$ be the substitution for each variable of formula. Then 

1. $\sigma$ is a mgu for $\varphi$;

2. $\{\sigma\}$ forms a complete set of unifiers for $\varphi$.
\end{corollary}

\begin{corollary}
The logic $\mathbf{PM4}$ has a unitary type of unification.
\end{corollary}

\bibliography{PM_unif}

\end{document}